\documentclass[12pt]{amsart}
\usepackage{amsmath,amsfonts,amssymb,amsthm,amstext,pgf,graphicx,hyperref,verbatim,lmodern,textcomp,color,young,tikz}
\usetikzlibrary{decorations}
\usepackage[mathscr]{euscript}
\usetikzlibrary{decorations.markings}
\usetikzlibrary{arrows}
\usepackage{float}
\theoremstyle{plain}
\newtheorem{theorem}{Theorem}
\newtheorem{combinatorial proof}{Theorem}[combinatorial proof]

\newtheorem{rem}[theorem]{Remark}
\theoremstyle{definition}
\newtheorem{defn}[theorem]{Definition}

\newtheorem{exmp}[theorem]{Example}

\title{}
\begin{document}
\title [A MATRIX FOR COUNTING PATHS IN ACYCLIC COLORED DIGRAPHS ] {A MATRIX FOR COUNTING PATHS IN ACYCLIC COLORED DIGRAPHS}
\author[Sudip Bera]{Sudip Bera}
\address[Sudip Bera]{Faculty of Mathematics, DA-IICT, Gandhinagar-382007, India.}
\email{sudip\_bera@daiict.ac.in}
\maketitle
\begin{abstract}
In this paper, we generalize a theorem of R. P. Stanley regarding the enumeration of paths in acyclic digraphs. 
\end{abstract}
\section{Introduction}
The evaluation of determinants is a nice topic, and fascinating for many people \cite{27,26,25,39,36}. Moreover, the evaluation of determinants in a combinatorial way gives more insight into `why' rather than `how'. This is especially true when the entries of a matrix have natural graph theoretic descriptions. Recently, in \cite{18} the author shows a combinatorial interpretation of the determinant of
a matrix as a generating function over Brauer diagrams in two different but related ways. This interpretation naturally explains why the determinant of an even anti-symmetric matrix is the square of a Pfaffian. In \cite{hook-sudip}, the author gives a combinatorial interpretation of product formulas of various types of hook determinants by counting weighted paths in a digraph. Moreover in \cite{22}, Doron Zeilberger gives combinatorial proofs of some determinantal identities. In \cite{10}, R. P. Stanley defines a matrix $A$ associated with an acyclic digraph $\Gamma$ (without multiple edges) such that the determinant of $A$ enumerates the paths in $\Gamma.$ In fact, the theorem of R. P. Stanley is as follows:

Let $[n]=\{1, 2, \cdots, n\}$ be the vertex set of $\Gamma.$ Since $\Gamma$ is acyclic, we may assume  without loss of generality that $i<j$ whenever $(i, j)$ is an edge. Let $\{ x_1, x_2, \cdots, x_n\}$ be a set of indeterminates. Construct an $n\times n$ matrix $A=(a_{ij})$ as follows: 
\begin{equation}\label{def:matrix stanely}
a_{ij}=
\begin{cases}
1+x_i, & \text{ if } i=j\\
x_i, & \text{ if } i<j \text{ and } (i, j) \text{ is not an edge }\\
0, & \text{ if } i<j \text{ and } (i, j) \text{ is an edge }  \\
x_i, & \text{ if } i>j.
\end{cases}
\end{equation}
Then the theorem of Stanley says the following:
\begin{theorem}[\cite{10}]\label{thm: Stanley enumeration of path in adigraph}
${\rm det(A)}=1+\sum\limits_Px_{k_1}x_{k_2}\cdots x_{k_j}$, where $P$ ranges over all paths $k_1k_2\cdots k_j.$
\end{theorem}
The above theorem is an easy consequence of a more general theorem of Goulden and Jackson \cite{Goulden-Jackson-79}. Motivated by Theorem \ref{thm: Stanley enumeration of path in adigraph}, in this article we define a matrix $A_{\Gamma_{k, R}}$ associated with an acyclic digraph $\Gamma_{k, R}$ (having multiple edges) such that $\text{det}(A_{\Gamma_{k, R}})$ enumerates the paths in the  digraph $\Gamma_{k, R}.$  
\section{Basic definitions and main theorem}
In this section, we state our main theorem. Before that, let us briefly describe a well-known graphical interpretation of determinant. See \cite{16} for details. To an $n\times n$ matrix $A,$ we can associate a weighted digraph $D(A),$ with vertex set $[n]$ and for each ordered pair $(i, j),$ there is an edge directed from $i$ to $j$ with weight $a_{ij}$. A \emph{linear subdigraph} $\gamma$ of $D(A)$ is a spanning collection of pairwise vertex-disjoint cycles. A loop around a single vertex is also considered to be a cycle of length $1.$ The \emph{weight} of a linear subdigraph $\gamma,$ written as $w(\gamma)$ is the product of the weights of all its edges. The number of cycles contained in $\gamma$ is denoted by $c(\gamma).$ The \emph{sign} of a linear subdigraph $\gamma$ is $(-1)^{n+c(\gamma)},$ where $n$ is the order of the matrix $A.$ Now the cycle-decomposition of permutations yields the following description of ${\rm det}(A),$ namely \[{\rm det}(A)=\sum\limits_{\gamma}(-1)^{n+c(\gamma)}w(\gamma),\] where the summation runs over all linear subdigraphs $\gamma$ of $D(A).$

The \emph{standard representation} of a directed cycle is defined to be a sequence of vertices $v_0\rightarrow v_1 \rightarrow \cdots \rightarrow v_k \rightarrow v_0$ such that $v_0$ is the smallest among the elements $\{ v_0, v_1, \cdots, v_k\}.$ Although any vertex on a directed cycle can be treated as the initial and terminal point of the cycle, in the standard representation only the smallest vertex is treated to be the initial and the terminal point of the cycle. A path $v_i\rightarrow v_{k_1}\rightarrow v_{k_2} \rightarrow \cdots \rightarrow v_{k_l} \rightarrow v_j$ in a graph $\Gamma$ is said to be \emph{vertex-increasing} if $v_i<v_{k_1}< \cdots < v_{k_{\ell}} <v_j.$ Vertex-increasing path between the vertices $v_i, v_j$ is denoted by $v_i\text{I}v_j.$ Similarly we may define a \emph{vertex-decreasing} path and denote the vertex-decreasing path between the vertices $v_i, v_j$ by $v_i\text{D}v_j.$ A path is said to be \emph{IDI} path if the standard representation of this path contains a subpath of the form $v_i\text{I}v_j\text{D}v_k\text{I}v_{\ell},$ for some $v_i, v_j, v_k, v_{\ell}.$ A vertex $v$ of a cycle $C_i$ in a linear subdigraph is said to be \emph{enclosed} if it is surrounded by another cycle $C_j (\neq C_i)$ of that linear subdigraph.

For example, in Figure \ref{fig:singularity on stanly's thm}, there are four linear subdigraphs with $4$ vertices. In the first linear subdigraph vertices $2$ and $3$ are surrounded by the cycle $1\rightarrow 4\rightarrow 1.$ So, $2, 3$ are enclosed vertices. Similarly, in the second linear subdigraph $2$ and $3$ are enclosed vertices. Infact, here $2$ and $3$ are surrounded by the cyles $1\rightarrow 3\rightarrow 1$ and $2\rightarrow 4\rightarrow 2$ respecively.  
\begin{defn}\label{defn:complex lsd}
A linear subdigraph is said to be \emph{complex} if and only if
\begin{enumerate}
\item 
either at least one of its cycle encloses a vertex of another nontrivial (i.e., not a loop) cycle
\item
or it has at least one cycle $C,$ whose standard representation contains a subpath of the form $v_i\text{I}v_j\text{D}v_k\text{I}v_{\ell},$ (i.e., $C$ has a IDI path) for some $v_i, v_j, v_k, v_{\ell}.$	
\end{enumerate}
\end{defn}
% \begin{defn}\label{defn: singular and point of singularity of lsd}
A cycle $C$ of a complex linear subdigraph is said to be \emph{singular} if and only if either $C$ encloses a vertex of another nontrivial cycle of the same linear  subdigraph, or it has a IDI path (of the form $v_i\text{I}v_j\text{D}v_k\text{I}v_{\ell},$ for some $v_i, v_j, v_k, v_{\ell}).$	
%\end{enumerate}
%\end{defn}
\begin{defn}\label{defn: point of singularity}
A vertex $v_k$ is said to be a point of \emph{singularity} of a singular cycle $C$ if $v_k$ is a vertex of another nontrivial cycle of the same linear subdigraph enclosed by $C$ or there exists a IDI path $v_i\text{I}v_j\text{D}v_k\text{I}v_{\ell}$ along $C$ (for the latter case, sometime we call $v_k$ a \emph{corner} point of the IDI path).
\end{defn}
\begin{figure}[H]
	\tiny
	\tikzstyle{ver}=[]
	\tikzstyle{vert}=[circle, draw, fill=black!100, inner sep=0pt, minimum width=4pt]
	\tikzstyle{vertex}=[circle, draw, fill=black!.5, inner sep=0pt, minimum width=4pt]
	\tikzstyle{edge} = [draw,thick,-]
	%\tikzstyle{edge_style} = [draw=black, line width=2, ultra thick]
	\tikzstyle{node_style} = [circle,draw=blue,fill=blue!20!,font=\sffamily\Large\bfseries]
	\centering
	%\tikzset{->,>=stealth', auto,node distance=1cm,
	%	thick,main node/.style={circle,draw,font=\sffamily\Large\bfseries}}
	\newcommand{\updownarrows}{\mathbin\uparrow\hspace{-.5em}\downarrow}
	\newcommand{\downuparrows}{\mathbin\downarrow\hspace{-.5em}\uparrow}
	
	\tikzset{->-/.style={decoration={
				markings,
				mark=at position #1 with {\arrow{>}}},postaction={decorate}}}
	\begin{tikzpicture}[scale=1]
	\tikzstyle{edge_style} = [draw=black, line width=2mm, ]
	\tikzstyle{node_style} = [draw=blue,fill=blue!00!,font=\sffamily\Large\bfseries]
	%EdgeStyle/.append style = {->, bend left} }
	\tikzset{
		LabelStyle/.style = { rectangle, rounded corners, draw,
			minimum width = 2em, fill = yellow!50,
			text = red, font = \bfseries },
		VertexStyle/.append style = { inner sep=5pt,
			font = \Large\bfseries},
		EdgeStyle/.append style = {->, bend left} }
	\tikzset{vertex/.style = {shape=circle,draw,minimum size=1.5em}}
	\tikzset{edge/.style = {->,> = latex'}}
%%%%%%%%%%%%%%%%%%%%%%%%%%%%%%%%%%%%%%%%%%%%%%%%%%%%%%%%%%%%%%%%%%%%
	\node[] (a) at  (0,0) {$\bf{1}$};
	\node[] (b) at  (1.5, 0) {$\bf{2}$};
	\node[] (c) at  (3,0) {$\bf{3}$};
	\node[] (d) at  (4.5,0) {$\bf{4}$};
	%%%%%% fiiliing of node in first lsd %%%%%%
    \fill[blue!100!](.2,0) circle (.08);
    \fill[blue!100!](4.3,0) circle (.08);
    \draw[edge, blue] (.2,0)  to[bend left] (4.3, 0);
	\draw[edge, blue] (4.3, 0)  to[bend left] (.2, 0);
	%%%%%%%%%%%% red lsd %%%%%%%%%%%%%%
	\draw[edge, red] (1.7,0)  to[bend left] (2.8,0);
	\draw[edge, red] (2.8,0)  to[bend left] (1.7,0);
     \fill[red!100!](1.7,0) circle (.08);
	\fill[red!100!](2.8,0) circle (.08);
	%%%%%%%%%%%%%%%%%%%%%%%%%%%%%%%%%%%%%% 2nd  lsd %%%%%%%%%%%%%%%%%%%%%%%%%%%%%%%%
	\node[] (a) at  (7,0) {$\bf{1}$};
	\node[] (b) at  (8.5, 0) {$\bf{2}$};
	\node[] (c) at  (10,0) {$\bf{3}$};
	\node[] (d) at  (11.5,0) {$\bf{4}$};
	%%%%%% fiiliing of node in first lsd%%%%%%%%
	\fill[blue!100!](7.2,0) circle (.08);
	\fill[red!100!](11.3,0) circle (.08);
	\draw[edge, blue] (7.2, 0)  to[bend left] (9.8, 0);
	\draw[edge, blue] (9.8, 0)  to[bend left] (7.2, 0);
	%%%%%%%%%%%% red lsd %%%%%%%%%%%%%%%%%%%%
	\draw[edge, red] (8.7,0)  to[bend left] (11.3,0);
	\draw[edge, red] (11.3,0)  to[bend left] (8.7, 0);
	\fill[red!100!](8.7,0) circle (.08);
	\fill[blue!100!](9.8,0) circle (.08);
	%%%%%%%%%%%%%%%%%%%%%%%%%%%%%%%%%%%%%% 3rd  lsd %%%%%%%%%%%%%%%%%%%%%%%%%%%%%%%%%%%
	\node[] (a) at  (0,-3) {$\bf{1}$};
	\node[] (b) at  (1.5, -3) {$\bf{2}$};
	\node[] (c) at  (3,-3) {$\bf{3}$};
	\node[] (d) at  (4.5,-3) {$\bf{4}$};
	%%%%%% fiiliing of node in first lsd %%%%%%
	\fill[blue!100!](.2,-3) circle (.08);
	\fill[blue!100!](4.3,-3) circle (.08);
	\draw[edge, blue] (.2, -3)  to[bend left] (4.3, -3);
	\draw[edge, blue] (4.3, -3)  to[bend left] (1.7, -3);
	\draw[edge, blue] (1.7, -3)  to[bend left] (2.8, -3);
	\draw[edge, blue] (2.8,-3)  to[bend left] (.2, -3);
	%%%%%%%%%%%% red lsd %%%%%%%%%%%%%%
	%\draw[edge, red] (b)  to[bend left] (c);
	%\draw[edge, red] (c)  to[bend left] (b);
	\fill[blue!100!](1.7,-3) circle (.08);
	\fill[blue!100!](2.8,-3) circle (.08);
	%%%%%%%%%%%%%%%%%%%%%%%%%%%%%%%%%%%%%% 4th  lsd %%%%%%%%%%%%%%%%%%%%%%%%%%%%%%%%
	\node[] (a) at  (7,-3) {$\bf{1}$};
	\node[] (b) at  (8.5, -3) {$\bf{2}$};
	\node[] (c) at  (10,-3) {$\bf{3}$};
	\node[] (d) at  (11.5,-3) {$\bf{4}$};
	%%%%%% fiiliing of node in first lsd%%%%%%%%
	\fill[blue!100!](7.2,-3) circle (.08);
	\fill[blue!100!](11.3,-3) circle (.08);
	\draw[edge, blue] (7.2,-3)  to[bend left] (9.8,-3);
	\draw[edge, blue] (9.8,-3)  to[bend left] (8.7,-3);
	\draw[edge, blue] (8.7,-3)  to[bend left] (11.3,-3);
	\draw[edge, blue] (11.3,-3)  to[bend left] (7.2, -3);
	%%%%%%%%%%%% red lsd %%%%%%%%%%%%%%%%%%%%
	%\draw[edge, red] (b)  to[bend left] (d);
	%\draw[edge, red] (d)  to[bend left] (b);
	\fill[blue!100!](8.7,-3) circle (.08);
	\fill[blue!100!](9.8,-3) circle (.08);
	
%	\draw[edge] (d)  to[bend left] (e);
	%\draw[edge] (e)  to[bend left] (c);
	\end{tikzpicture}
	\caption{Complex linear subdigraphs and points of singularity.}
	\label{fig:singularity on stanly's thm}	
\end{figure}
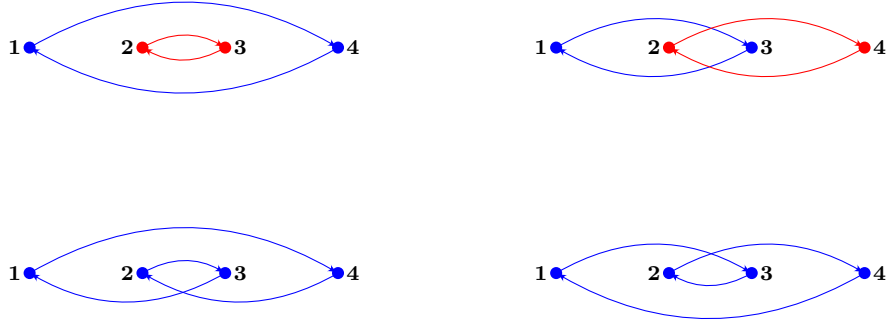
For example, in Figure \ref{fig:singularity on stanly's thm}, the vertex $2$ is a point of singularity. Infact, in the first two linear subdigraphs the vertex $2$ is a enclosed vertex but in the other two linear subdigraphs vertex $2$ is a corner point (since $1\rightarrow 4\rightarrow 2\rightarrow 3\rightarrow 1$ and $1\rightarrow 3\rightarrow 2\rightarrow 4\rightarrow 1$ are IDI paths).  
\begin{defn}\label{defn: k colored di graph}
A \emph{$k$-colored digraph}, denoted by $\Gamma_{k, R}$ is a digraph equipped with a finite set $R=\{c_1, c_2, \cdots, c_k\}$ called the set of \emph{colors} such that for any ordered pair of vertices $(i, j)$ in $\Gamma_{k, R},$ there are $\ell (0\leq\ell\leq k)$ directed edges from $i$ to $j$ each receiving distinct colors from $R$ (i.e., no two of the directed edges from $i$ to $j$ receive the same color). 
\end{defn}
A $k$-colored digraph $\Gamma_{k, R}$ is said to be \emph{acyclic} if it has no cycle.  In this case, without loss of generality we assume that $i < j$ whenever there is an edge from $i$ to $j.$ i.e., in $\Gamma_{k, R}$ there is no edge from the vertex $j$ to the vertex $i,$ for any $i<j.$ By $v_{k_1}\underbrace{\rightarrow }_{c_{r_1}}v_{k_2}\underbrace{\rightarrow }_{c_{r_2}}v_{k_3}\underbrace{\rightarrow}_{c_{r_3}}v_{k_4}\underbrace{\rightarrow}_{c_{r_4}} \cdots\underbrace{\rightarrow}_{c_{r_{t-2}}} v_{k_{t-1}}\underbrace{\rightarrow}_{c_{r_{t-1}}}v_{k_{t}},$ we mean a color (not necessarily distinct) path from the vertex $v_{k_1}$ to the vertex $v_{k_t}$ such that the edge from $v_{k_i}$ to $v_{k_{i+1}}$ is colored by the color $c_{r_i},$ for $i=1, \cdots, t-1.$ Let the vertex set of the graph $\Gamma_{k, R}$ be $[n].$ For each $i\in [n],$ we assign $k$ variables $x_i^{(1)}, x_i^{(2)}, \cdots, x_i^{(k)},$ where for each $j\in [k], x_i^{(j)}$ is the variable corresponding to the color $c_j.$ Let $\{x_1^{(1)}, x_1^{(2)}, \cdots, x_1^{(k)},  x_2^{(1)}, x_2^{(2)}, \cdots, x_2^{(k)}, \cdots, x_n^{(1)}, x_n^{(2)}, \cdots, x_n^{(k)}\}$ be the set of variables.
Let us consider an $n\times n$  matrix $A_{\Gamma_{k, R}}=(a_{ij})$ as follows:
\begin{equation}\label{def:matrix stanely gen}
a_{ij}=
\begin{cases}
1+\sum\limits_{t\in [k]} x_i^{(t)},& \text{ if } i=j\\
\sum\limits_{t\in [k]\setminus \{i_1, \cdots, i_r\}}x_i^{(t)},& \text{ if } i<j \text{ and }  i \text{ to } j \text{ there are } r (0\leq r< k)\text{ different  } \\& \text{ colored edges using the colors }{c_{i_1}}, \cdots, {c_{i_r}}\\
0,& \text{ if } i<j \text{ and }  i \text{ to } j \text{ there are } k \text{ different colored edges }\\
\sum\limits_{t\in [k]}x_i^{(t)},& \text{ if } i>j.  \\
\end{cases}
\end{equation}
Note that, if $i<j$ and $r=0,$ then the set $\{i_1, \cdots, i_r\}=\emptyset.$ So, in this case $a_{ij}=\sum\limits_{t\in [k]}x_i^{(t)}.$
With the notation and terminology introduced above we can state our main theorem.
\begin{theorem}\label{thm: Main thm Stanly's thm Gen}
Let $\Gamma_{k, R}$ be a $k$-colored acyclic digraph such that the vertex set of $\Gamma_{k, R}$ is $[n].$ Let $A_{\Gamma_{k, R}}$ be an $n\times n$ matrix defined as \eqref{def:matrix stanely gen}. Then $\text{det}(A_{\Gamma_{k, R}})=$
\begin{align*}
1+\sum\limits_P \left(x_{i_1}^{(r_1)}x_{i_2}^{(r_2)}\cdots x_{i_{t-1}}^{(r_{t-1})} x_{i_t}^{(1)}+x_{i_1}^{(r_1)}x_{i_2}^{(r_2)}\cdots x_{i_{t-1}}^{(r_{t-1})} x_{i_t}^{(2)}+\cdots+x_{i_1}^{(r_1)}x_{i_2}^{(r_2)}\cdots x_{i_{t-1}}^{(r_{t-1})} x_{i_t}^{(k)}\right),\\
\text{ where } P \text{ ranges over all paths } i_1\underbrace{\rightarrow}_{c_{r_1}} i_2\underbrace{\rightarrow}_{c_{r_2}} i_3\underbrace{\rightarrow}_{c_{r_3}} \cdots\underbrace{\rightarrow}_{c_{r_{t-2}}} i_{t-1}\underbrace{\rightarrow}_{c_{r_{t-1}}} i_t \text{ in the graph } \Gamma_{k, R}.
\end{align*}
\end{theorem}
\begin{rem}
Note that, if the color set $R$ contains exactly one color (that is if $R=\{c\},$ where $c$ represents a color), then  $\Gamma_{1, R}$ is acyclic digraph without multiple edges. Therefore as a corollary of this theorem, we get Theorem \ref{thm: Stanley enumeration of path in adigraph}.	
\end{rem}
In this portion, we explain our main theorem with one example.
\begin{exmp}
Here we first fix the color set $R=\{c_1, c_2, c_3\},$ where $c_1, c_2, c_3$ represent the colors black, blue, red respectively. 

We choose the variables $x_1^{(1)}, x_1^{(2)}, x_1^{(3)}, x_2^{(1)}, x_2^{(2)}, x_2^{(3)}, x_3^{(1)}, x_3^{(2)}, x_3^{(3)}, x_4^{(1)}, x_4^{(2)}, x_4^{(3)},$ where we use the variable $x_i^{(j)}$ for the edge coming out from the vertex $i$ and colored by $c_j$-th color, $j\in [3].$	
\begin{figure}[H]
\tiny
\tikzstyle{ver}=[]
\tikzstyle{vert}=[circle, draw, fill=black!100, inner sep=0pt, minimum width=4pt]
\tikzstyle{vertex}=[circle, draw, fill=black!.5, inner sep=0pt, minimum width=4pt]
\tikzstyle{edge} = [draw,thick,-]
%\tikzstyle{edge_style} = [draw=black, line width=2, ultra thick]
\tikzstyle{node_style} = [circle,draw=blue,fill=blue!20!,font=\sffamily\Large\bfseries]
\centering
\tikzset{->,>=stealth', auto,node distance=1cm,
	thick,main node/.style={circle,draw,font=\sffamily\Large\bfseries}}
\tikzset{->-/.style={decoration={
			markings,
			mark=at position #1 with {\arrow{>}}},postaction={decorate}}}
\begin{tikzpicture}[scale=1]
\tikzstyle{edge_style} = [draw=black, line width=2mm, ]
\tikzstyle{node_style} = [draw=blue,fill=blue!00!,font=\sffamily\Large\bfseries]
%EdgeStyle/.append style = {->, bend left} }
\tikzset{
	LabelStyle/.style = { rectangle, rounded corners, draw,
		minimum width = 2em, fill = yellow!50,
		text = red, font = \bfseries },
	VertexStyle/.append style = { inner sep=5pt,
		font = \Large\bfseries},
	EdgeStyle/.append style = {->, bend left} }
\tikzset{vertex/.style = {shape=circle,draw,minimum size=1.5em}}
\tikzset{edge/.style = {->,> = latex'}}		
%\draw[->, line width=.2 mm] (7.5,-3)--(7.5,-1);
		\draw[->, line width=.2 mm] (7.5,-1)--(6,.92);
		\draw[->, line width=.2 mm] (7.5,-1)--(9,.92);
		\draw[->, line width=.2 mm] (7.5,-3)--(7.5,-1.05);
		\node (a) at (7.1,-1){$\bf{v_2}$};
		\node (b) at (7.5, -3.2){$\bf{v_1}$};
		\node (c) at (6,1.2){$\bf{v_3}$};
		\node (d) at (9,1.2){$\bf{v_4}$};
		\fill[black!100!] (7.5,-1) circle (.08);
		\fill[black!100!] (7.5, -3) circle (.08);
		\fill[black!100!] (6,.92) circle (.08);
		\fill[black!100!] (9,.92) circle (.08);
		\draw[edge,blue] (7.5,-3)  to[bend right] (7.53,-1.05);
\end{tikzpicture}
\caption{A colored digraph $\Gamma_{3, R}$}
\label{matrix example stanley thm gen}
\end{figure}
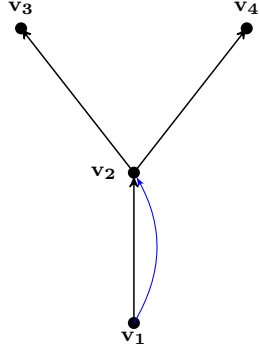
Now from expression \eqref{def:matrix stanely gen}, the matrix corresponding to the graph $\Gamma$ in Figure \ref{matrix example stanley thm gen} is the following:
\[A_{\Gamma_{3, R}}=\left(
\begin{array}{cccc}
1+\sum\limits_{i=1}^3x_1^{(i)} & x_1^{(3)}& \sum\limits_{i=1}^3x_1^{(i)}& \sum\limits_{i=1}^3x_1^{(i)}\\ 
\sum\limits_{i=1}^3x_2^{(i)} & 1+\sum\limits_{i=1}^3x_2^{(i)}& \sum\limits_{i=2}^3x_2^{(i)} & \sum\limits_{i=2}^3x_2^{(i)} \\
\sum\limits_{i=1}^3x_3^{(i)} & \sum\limits_{i=1}^3x_3^{(i)} & 1+\sum\limits_{i=1}^3x_3^{(i)} &\sum\limits_{i=1}^3x_3^{(i)} \\
\sum\limits_{i=1}^3x_4^{(i)} & \sum\limits_{i=1}^3x_4^{(i)}& \sum\limits_{i=1}^3x_4^{(i)}& 1+\sum\limits_{i=1}^3x_4^{(i)}
\end{array}
\right).\]
$\text{ Then, }\text{det}(A_{\Gamma_{3, R}})$ is equal to
\begin{align*} 
1+&x_1^{(1)}+x_1^{(2)}+x_1^{(3)}+x_2^{(1)}+x_2^{(2)}+x_2^{(3)}+x_3^{(1)}+x_3^{(2)}+x_3^{(3)}
+x_4^{(1)}+x_4^{(2)}+x_4^{(3)}\\+&x_1^{(1)}x_2^{(1)}+x_1^{(1)}x_2^{(2)}+x_1^{(1)}x_2^{(3)}+x_1^{(2)}x_2^{(1)}+x_1^{(2)}x_2^{(2)}+x_1^{(2)}x_2^{(3)}
+x_2^{(1)}x_3^{(1)}+x_2^{(1)}x_3^{(2)}+x_2^{(1)}x_3^{(3)}\\+&x_2^{(1)}x_4^{(1)}+x_2^{(1)}x_4^{(2)}+x_2^{(1)}x_4^{(3)}
+x_1^{(1)}x_2^{(1)}x_3^{(1)}+x_1^{(1)}x_2^{(1)}x_3^{(2)}+x_1^{(1)}x_2^{(1)}x_3^{(3)}+x_1^{(2)}x_2^{(1)}x_3^{(1)}\\+&x_1^{(2)}x_2^{(1)}x_3^{(2)}+x_1^{(2)}x_2^{(1)}x_3^{(3)}
+x_1^{(1)}x_2^{(1)}x_4^{(1)}+x_1^{(1)}x_2^{(1)}x_4^{(2)}+x_1^{(1)}x_2^{(1)}x_4^{(3)}\\+&x_1^{(2)}x_2^{(1)}x_4^{(1)}+x_1^{(2)}x_2^{(1)}x_4^{(2)}+x_1^{(2)}x_2^{(1)}x_4^{(3)}.
\end{align*}
\end{exmp}
\section{Proof of main theorem}
We now prove our main theorem in a combinatorial way.  
\begin{proof}[Proof of Theorem \ref{thm: Main thm Stanly's thm Gen}]
Take \[W=\{x_1^{(1)}, x_1^{(2)}, \cdots, x_1^{(k)},  x_2^{(1)}, x_2^{(2)}, \cdots, x_2^{(k)}, \cdots, x_n^{(1)}, x_n^{(2)}, \cdots, x_n^{(k)}\}\] as the set of \emph{letters}. The \emph{free monoid} $W^*$ is the set of all finite sequences (including the empty sequence, denoted by $1$) of elements of $W,$ usually called \emph{words}, with the operation of concatenation. Construct an algebra  from $W^*$ by taking formal sum of elements of $W$ with coefficient in $\mathbb{Z},$ extending the multiplication by usual distributivity. For example, in this algebra,
\begin{align*}
(x_i^{(1)}+x_j^{(2)})(x_i^{(1)}+x_j^{(2)})=&  (x_i^{(1)}+x_j^{(2)})^2= x_i^{(1)}x_i^{(1)}+x_i^{(1)}x_j^{(2)}+x_j^{(2)}x_i^{(1)}+x_j^{(2)}x_j^{(2)},\\ &(1+x_i^{(1)})x_j^{(2)}=x_j^{(2)}+x_i^{(1)}x_j^{(2)},  \text{ etc}.
\end{align*}
A nonempty word  $x_{k_1}^{(i_1)}x_{k_2}^{(i_2)}\cdots x_{k_m}^{(i_m)}$ is said to be \emph{nice} if $k_1<k_2<\cdots<k_m.$ The empty word $1$ is also considered to be nice. Now let $B$ be a set of ordered pairs $(x_{k_r}^{(i_r)}, x_{k_s}^{(i_s)})$ with $k_r<k_s.$ In terms of digraph, $(x_{k_r}^{(i_r)}, x_{k_s}^{(i_s)})\in B$ if and only if there is no $i_r$-color edge (an edge colored by $i_r$th-color) from the vertex $k_r$ to the vertex $k_s.$ Therefore, $(x_{k_r}^{(i_r)}, x_{k_s}^{(i_s)})\in B$ implies that each of the following pair \[(x_{k_r}^{(i_r)}, x_{k_s}^{(1)}), (x_{k_r}^{(i_r)}, x_{k_s}^{(2)}), \cdots, (x_{k_r}^{(i_r)}, x_{k_s}^{(k)})\] is also a member of the set $B.$ We call a word $w\in W^*$ \emph{best} if $w$ is nice and no member of the set $B$ can occur in $w$ as a consecutive pair. Let $S$ be the sum of all best words in $W^*.$ Note that $1$ is also a best word. Now we derive an expression for $S$ by using the \emph{Principle of Inclusion Exclusion}(in short \emph{PIE}). The sum of all possible nice words is 
\begin{align}\label{single term w0}
W_0=\prod\limits_{g=1}^{n}\left(1+\sum\limits_{a=1}^kx_{g}^{(a)}\right).
\end{align}
The sum of all possible nice words, where there is an occurrence of at least one pair $(x_r^{(i)}, x_s^{(j)})\in B$ as a consecutive pair is
\begin{align}\label{ONE arbritrary MEMBER OF W_1, x_r^{i_r}x_s^{i_s}}
w^{(i, j)}_{rs}=\prod\limits_{g=1}^{r-1}\left(1+\sum\limits_{a=1}^kx_{g}^{(a)}\right)x_r^{(i)}x_s^{(j)}\prod\limits_{g=s+1}^{n}\left(1+\sum\limits_{a=1}^kx_{g}^{(a)}\right).
\end{align}
Let $W_1$ be the sum of all possible nice words, where there is an occurrence of at least one pair belongs to the set $B$ as a consecutive pair.
The sum of all possible nice words, where there is an occurrence of at least two pairs $(x_p^{(i)}, x_q^{(j)}), (x_r^{(\ell)}, x_s^{(m)})\in B$ (where $(p<q, r<s)$ and $i, j, \ell, m\in [k]$) as two consecutive pairs is 
\begin{align}\label{two arbritrary MEMBER OF W_2, x_r^{i_r}x_s^{i_s}, pq}
w^{(i, j, \ell, m)}_{pq, rs}=\prod\limits_{g=1}^{p-1}\left(1+\sum\limits_{a=1}^kx_{g}^{(a)}\right)x_p^{(i)} x_q^{(j)}\prod\limits_{g=q+1}^{r-1}\left(1+\sum\limits_{a=1}^kx_{g}^{(a)}\right)x_r^{(\ell)} x_s^{(m)}\prod\limits_{g=s+1}^{n}\left(1+\sum\limits_{a=1}^kx_{g}^{(a)}\right).
\end{align}
Let $W_2$ be the sum of all possible nice words, where there is an occurrence of at least two pairs belong to the set $B$ as consecutive pairs.
Define $W_3, W_4,\cdots$ etc., in a similar fashion. Then by the PIE, \[S=W_0-W_1+W_2-W_3+\cdots.\] Now, to any term of an arbitrary $W_i$ in the above expression, we associate a term or terms of $\text{det}(A_{\Gamma_{k, R}})$ so that the weights and signs match. Actually we want to show that the image of a word under the ring homomorphism \[f: W^*\rightarrow \mathbb{Z}\left[x_1^{(1)},  \cdots, x_1^{(k)}, \cdots, x_n^{(1)}, \cdots, x_n^{(k)}\right], f(x_i^{(j)})=x_i^{(j)} \text{ (for all } i, j) \]  is the sum of the weights of all linear subdigraphs associated to that word. This can be accomplished in the following way.
	
For the sake of simplicity, here we demonstrate, how to associate to each term of $W_0$ and $ W_1$ in the above expression, a term or terms of $\text{det}(A_{\Gamma_{k, R}}).$ The case of other indices $W_2, W_3, W_4, \cdots$ are similar. First of all $W_0$ consists the single term 
\begin{align}\label{w(W_0)}
\prod\limits_{g=1}^{n}\left(1+\sum\limits_{a=1}^kx_{g}^{(a)}\right).
\end{align}
To this term, we associate the linear subdigraph consisting of $n$ disjoint loops around each vertex. Note that \[a_{ii}=\left(1+\sum\limits_{a=1}^kx_{i}^{(a)}\right), \text{ for each } i\in [n].\] Consequently, the sign weight of this linear subdigraph is same as \eqref{w(W_0)}. Now let us assume  that in the digraph $\Gamma,$ there are exactly $m$-colored edges from the vertex $r$ to the vertex $s$ receiving $i_1, \cdots, i_m$ different colors. Hence from the vertex $r$ to the vertex $s,$ there is no $i$-colored edge, where $i\in [k]\setminus \{i_1 \cdots, i_m\}.$ As a result, in the matrix $A_{\Gamma_{k, R}}=(a_{ij})_{n\times n},$ $a_{rs}=\sum\limits_{\ell=m+1}^k x_r^{(i_{\ell})}.$
Consider the set 
\begin{align}\label{W_1(1,j)}
W_1(r,s)=\{\prod\limits_{i=1}^{r-1}\left(1+\sum\limits_{a=1}^kx_{i}^{(a)}\right)x_r^{(i_{\ell})}x_s^{(d)}\prod\limits_{i=s+1}^{n}\left(1+\sum\limits_{a=1}^kx_{i}^{(a)}\right): \ell\in [k]\setminus [m], d\in [k]\}.
\end{align}
Clearly, each member of the set $W_1(r,s)$ is a typical term in $W_1.$ The sign weighted sum of all elements in the set ${W_1(r,s)}$ (defined as \eqref{W_1(1,j)})is
\begin{equation}\label{w(W_(1, j))}
-\prod\limits_{i=1}^{r-1}\left(1+\sum\limits_{a=1}^kx_{i}^{(a)}\right)\sum\limits_{\ell=m+1}^k\sum\limits_{d=1}^k x_r^{(i_{\ell})}x_s^{(d)}\prod\limits_{i=s+1}^{n}\left(1+\sum\limits_{a=1}^kx_{i}^{(a)}\right).
\end{equation}
Now to \eqref{w(W_(1, j))}, we associate $2^{(c-1)}$ linear subdigraphs such that the sum of sign weights of $2^{(c-1)}$ linear subdigraphs is \eqref{w(W_(1, j))}, where $c=s-r.$ First consider the set \[T=\{r+1, r+2, \cdots, r+c-1\}\subset [n].\] Let $T'=\{t_1, \cdots, t_m\}$ be an arbitrary subset of $T$ such that $t_1<t_2<\cdots<t_m.$ Here we associate a linear subdigraph $L_{T'}$ as follows;

$L_{T'}$ consists of a cycle $r\rightarrow s\rightarrow t_m\rightarrow t_{m-1}\rightarrow \cdots\rightarrow t_2\rightarrow t_1\rightarrow r$ and $(n-m-2)$ disjoint loops around the remaining $(n-m-2)$ vertices. So, for each non empty  subset $T',$ we get a unique linear subdigraph $L_{T'}$ and 
\begin{align*}
&w(L_{T'})=(-1)^{m+1}\left(a_{rs}a_{st_m}a_{t_mt_{m-1}}\cdots a_{t_2t_1}a_{t_1r}\right)\prod\limits_{i\in [n]\setminus T} a_{ii}, \\
=(-1)^{m+1}&\left(\sum\limits_{\ell=m+1}^kx_r^{(i_{\ell})}\sum\limits_{a=1}^kx_s^{(a)}\sum\limits_{a=1}^kx_{t_m}^{(a)}\cdots\sum\limits_{a=1}^kx_{t_2}^{(a)}\sum\limits_{a=1}^kx_{t_1}^{(a)}\right)\prod\limits_{i\in [n]\setminus T}\left(1+\sum\limits_{a=1}^kx_{i}^{(a)}\right).
\end{align*}
If $T'=\emptyset,$ then $L_{T'}$ consists of a cycle $r\rightarrow s\rightarrow r$ and $(n-2)$ disjoint loops around the remaining $(n-2)$ vertices and
\begin{align*}
&w(L_{\emptyset})=-a_{rs}a_{sr}\prod\limits_{i\in [n]\setminus \{r, s\}}a_{ii}\\
=-&\sum\limits_{\ell=m+1}^kx_r^{(i_{\ell})}\sum\limits_{a=1}^kx_s^{(a)}\prod\limits_{i\in [n]\setminus \{r, s\}}\left(1+\sum\limits_{a=1}^kx_{i}^{(a)}\right)
\end{align*} 
Now the sum of the sign weights of all these $2^{(c-1)}$ linear subdigraphs is 
\begin{equation}\label{w(W_(r, s)),2^{c-1}}
-\prod\limits_{i=1}^{r-1}\left(1+\sum\limits_{a=1}^kx_{i}^{(a)}\right)\left(\sum\limits_{\ell=m+1}^k\sum\limits_{d=1}^k x_r^{(i_{\ell})}x_s^{(d)}\right)X\prod\limits_{i=s+1}^{n}\left(1+\sum\limits_{a=1}^kx_{i}^{(a)}\right),
\end{equation}
\text{ where }
\begin{align*}
&X= \prod\limits_{g=r+1}^{s-1}\left(1+\sum\limits_{a=1}^kx_{g}^{(a)}\right)-\sum\limits_{i\in T}\prod\limits_{g=r+1}^{i-1}\left(1+\sum\limits_{a=1}^kx_{g}^{(a)}\right)\left(\sum\limits_{a=1}^kx_{i}^{(a)}\right)\prod\limits_{g=i+1}^{s-1}\left(1+\sum\limits_{a=1}^kx_{g}^{(a)}\right)\\
&+\sum\limits_{i<j\in T}\prod\limits_{g=1}^{i-1}\left(1+\sum\limits_{a=1}^kx_{g}^{(a)}\right)\left(\sum\limits_{a=1}^kx_{i}^{(a)}\right)\prod\limits_{g=i+1}^{j-1}\left(1+\sum\limits_{a=1}^kx_{g}^{(a)}\right)\left(\sum\limits_{a=1}^kx_{j}^{(a)}\right)\prod\limits_{g=j+1}^{s-1}\left(1+\sum\limits_{a=1}^kx_{g}^{(a)}\right)\\&-\cdots+(-1)^{s-r-1}\prod\limits_{g= r+1}^{s-1}\left(\sum\limits_{a=1}^kx_{g}^{(a)}\right).
\end{align*}
Let $\mathcal{W}=\{x_{r+1}^{(1)},\cdots, x_{r+1}^{(k)}, x_{r+2}^{(1)},\cdots, x_{r+2}^{(k)},\cdots, x_{s-1}^{(1)},\cdots, x_{s-1}^{(k)}\}$ be the set of letters. Since the above expression is the PIE model for the sum of all nice words in $\mathcal{W}^*$ consisting of no letters, $X$ is precisely $1.$ 

For example, if $r=1$ and $s=2,$ then $W_1(1,2)$ contains the following terms
\begin{align*}
x_1^{(i_{m+1})}x_2^{(1)}&\prod\limits_{g=3}^{n}\left(1+\sum\limits_{a=1}^kx_{g}^{(a)}\right),\cdots, x_1^{(i_{m+1})}x_2^{(k)}\prod\limits_{g=3}^{n}\left(1+\sum\limits_{a=1}^kx_{g}^{(a)}\right), \cdots,\\& x_1^{(i_{k})}x_2^{(1)}\prod\limits_{g=3}^{n}\left(1+\sum\limits_{a=1}^kx_{g}^{(a)}\right),\cdots, x_1^{(i_{k})}x_2^{(k)}\prod\limits_{g=3}^{n}\left(1+\sum\limits_{a=1}^kx_{g}^{(a)}\right). 
\end{align*}
It is easy to see that, the sign of each term of $W_1(1,2)$ in $S$ is $-.$ Moreover, all the members of the set $W_1(1,2)$ are some typical terms of $W_1.$ Now the sign sum of all these terms (terms of $W_1(1,2)$) is \[-\sum\limits_{\ell=m+1}^k\sum\limits_{d=1}^k x_1^{(i_{\ell})}x_2^{(d)}\prod\limits_{g=3}^{n}\left(1+\sum\limits_{a=1}^kx_{g}^{(a)}\right).\]
In this case, we associate the linear subdigraph, consisting of the cycle $1\rightarrow 2 \rightarrow 1$ and $(n-2)$ disjoint loops around the remaining $(n-2)$ vertices. So, the number of cycle in this linear subdigraph is $n-1.$ Therefore, the weight of this linear subdigraph is  \[\sum\limits_{\ell=m+1}^k\sum\limits_{d=1}^k x_1^{(i_{\ell})}x_2^{(d)}\prod\limits_{g=3}^{n}\left(1+\sum\limits_{a=1}^kx_{g}^{(a)}\right)\] (as in the matrix
$a_{12}=\sum\limits_{\ell=m+1}^kx_1^{(i_{\ell})}, a_{21}=\sum\limits_{d=1}^kx_2^{(d)}, \text{ and for } i\geq 3, a_{ii}=1+\sum\limits_{d=1}^k x_i^{(d)})$ and its sign in $\text{det}(A_{\Gamma_{k, R}})$ is $-.$ 

Now if $r=1$ and $s=4,$ then 
\[W_{1}(1, 4)=\{x_1^{(i_{\ell})}x_4^{(d)}\prod\limits_{g=5}^{n}\left(1+\sum\limits_{a=1}^kx_{g}^{(a)}\right): \ell\in \{m+1, \cdots, k\}, d\in [k]\}.\] Here each term of the set $W_{1}(1, 4)$ is a typical term of $W_1.$ Also the sign sum of all the terms in the set $W_{1}(1, 4)$ is
\begin{equation}\label{equation: for 1 to j=4}
-\sum\limits_{\ell=m+1}^k\sum\limits_{d=1}^k x_1^{(i_{\ell})}x_4^{(d)}\prod\limits_{g=5}^{n}\left(1+\sum\limits_{a=1}^kx_{g}^{(a)}\right).
\end{equation}
Now, to \eqref{equation: for 1 to j=4} we associate four linear subdigraphs such that the sum of the signed weights of the four linear subdigraphs is same as \eqref{equation: for 1 to j=4}. 

The first one consists of a cycle $1\rightarrow 4 \rightarrow 1$ and $(n-2)$ disjoint loops around the remaining $(n-2)$ vertices. The second one consists of a cycle $1\rightarrow 4 \rightarrow 2\rightarrow 1$ and $(n-3)$ disjoint loops around the remaining $(n-3)$ vertices. The third one consists of a cycle $1\rightarrow 4 \rightarrow 3\rightarrow 1$ and $(n-3)$ disjoint loops around the remaining $(n-3)$ vertices. The last one consists of a cycle $1\rightarrow 4 \rightarrow 3\rightarrow 2\rightarrow 1$ and $(n-4)$ disjoint loops around the remaining $(n-4)$ vertices. Note that the signed weights of these four linear subdigraphs are
\begin{align*}\label{weight: for 1 to j=4, n-2 loops,1}
&-\sum\limits_{\ell=m+1}^k\sum\limits_{d'=1}^k x_1^{(i_{\ell})}x_4^{(d')}\prod\limits_{g\in[k]\setminus\{1, 4\}}\left(1+\sum\limits_{a=1}^kx_{g}^{(a)}\right),\\&
+\sum\limits_{\ell=m+1}^k\sum\limits_{d'=1}^k\sum\limits_{c'=1}^k x_1^{(i_{\ell})}x_4^{(d')}x_2^{(c')}\prod\limits_{g\in[k]\setminus\{1, 2, 4\}}\left(1+\sum\limits_{a=1}^kx_{g}^{(a)}\right),\\&
+\sum\limits_{\ell=m+1}^k\sum\limits_{d'=1}^k\sum\limits_{c=1}^k x_1^{(i_{\ell})}x_4^{(d')}x_3^{(c')}\prod\limits_{g\in[k]\setminus\{1, 3, 4\}}\left(1+\sum\limits_{a=1}^kx_{g}^{(a)}\right),\\& -\sum\limits_{\ell=m+1}^k\sum\limits_{d'=1}^k\sum\limits_{c'=1}^k\sum\limits_{b'=1}^k x_1^{(i_{\ell})}x_4^{(d')}x_3^{(c')}x_2^{(b')}\prod\limits_{g\in[k]\setminus\{1, 2, 3, 4\}}\left(1+\sum\limits_{a=1}^kx_{g}^{(a)}\right)
\end{align*}
respectively.
The contribution of these four linear subdigraphs to $\text{det}(A_{\Gamma_{k, R}})$ is the same as \eqref{equation: for 1 to j=4}, since
\begin{align*}
 \left(1+\sum\limits _{a=1}^kx_{2}^{(a)}\right)&\left(1+\sum\limits _{a=1}^kx_{3}^{(a)}\right)-\left(\sum\limits _{a=1}^kx_{2}^{(a)}\right)\left(1+\sum\limits _{a=1}^kx_{3}^{(a)}\right)-\\&\left(1+\sum\limits _{a=1}^kx_{2}^{(a)}\right)\left(\sum\limits _{a=1}^kx_{3}^{(a)}\right)+\left(\sum\limits _{a=1}^kx_{2}^{(a)}\right)\left(\sum\limits _{a=1}^kx_{3}^{(a)}\right)=1. 
\end{align*}
So, we have associated some terms of $\text{det}(A_{\Gamma_{k, R}})$ to the terms of $S,$ matching signs and weights. But
in $\text{det}(A_{\Gamma_{k, R}}),$ beside these linear subdigraphs, that have been used so far, there are other linear subdigraphs, which are complex linear subdigraphs according to Definition \ref{defn:complex lsd}. Therefore, to prove the theorem, we have to prove that the contribution of the complex linear subdigraphs to $\text{det}(A_{\Gamma_{k, R}})$ is zero. We prove this by defining a weight preserving and sign-reversing involution on $\mathcal{C},$ where $\mathcal{C}$ is the set of all complex linear subdigraphs of the digraph $D(A_{\Gamma_{k, R}}).$ Now let us define the sign-reversing involution on the set of all complex linear subdigraphs. 

Choose a complex linear subdigraph. Choose the singular cycle with the smallest
initial point (in the standard representation). Find its smallest point of singularity. Either it is an enclosed point or a corner point of an IDI path, as described earlier.
Now, we follow the following algorithm:

Let $\gamma\in \mathcal{C}.$ Let the vertex $v_s$ be the smallest point of singularity, and it is an enclosed point. Then $v_s$ is the initial point of one nontrivial cycle $C_{\text{in}}$ (say). 
 Also, there is one nontrivial cycle $C_{\text{out}}$ (say) 
 such that the cycle $C_{\text{out}}$ encloses the vertex $v_s.$ Clearly, $v_s\geq 2$ (as $v_s$ is an enclosed vertex). Let $u_1$ be the initial vertex of the cycle $C_{\text{out}}.$ Then there is a maximal decreasing subpath $P$ from the vertex $u_j$ (say) to $u_1$ in $C_{\text{out}}.$ So, there is a vertex $u_m$ in $C_{\text{out}}$ such that $u_m<u_j$ and $u_m\rightarrow u_j.$ Based on the fact that $v_s$ is enclosed by the cycle $C_{\text{out}}$ there exist consecutive vertices $u_{\ell}, u_r$ in $C_{\text{out}}$ such that $u_1\leq u_{\ell}<v_s<u_r$ and $u_r\rightarrow u_{\ell}$ in the cycle $C_{\text{out}}.$ Now, $v_s$ is the initial vertex of the cycle $C_{\text{in}},$ then there exist vertices $v_k \text{ and }v_t$ in $C_{\text{in}}$ such that $v_k>v_s<v_t$ and $v_t\rightarrow v_s, v_s\rightarrow v_k$ in $C_{\text{in}}.$ 
 Therefore, the cycles $C_{\text{in}}$ and $C_{\text{out}}$ are of the following form:
 \begin{align*}
C_{\text{in}}:&v_s\rightarrow v_k\rightarrow\cdots\rightarrow v_t\rightarrow v_s\\
C_{\text{out}}: &u_1\rightarrow\cdots\rightarrow u_i\rightarrow\cdots\rightarrow u_m\rightarrow u_j\rightarrow \cdots\rightarrow u_r\rightarrow u_{\ell}\rightarrow \cdots\rightarrow u_1.
\end{align*}
First, we delete the arcs $v_t\rightarrow v_s$ and $u_r\rightarrow u_{\ell}.$ Now from $u_1$ to $u_r$ follow the path along the enclosing cycle $C_{\text{out}},$ then move from $u_r$ to $v_s$ by adding the arc $u_r\rightarrow v_s$ (this is possible as $v_s<u_r$ implies that $a_{u_rv_s}\neq 0$ in $A(\Gamma_{k, R}))$ to reach the point $v_s$ and meet the cycle $C_{\text{in}}.$ Then follow the path from $v_s$ to $v_t$ along the cycle $C_{\text{in}}$ and then move from $v_t$ to $u_{\ell}$ by adding the arc $v_t\rightarrow u_{\ell}$ (since $u_{\ell}<v_t$ implies that $a_{v_tu_{\ell}}\neq 0$ in $A(\Gamma_{k, R})).$ After that from $u_{\ell}$ to $u_1,$ moving along the enclosing cycle $C_{\text{out}}.$ As a result, we get the following cycle:
\begin{align*}
C_{\text{combine}}: u_1\rightarrow\cdots \rightarrow u_i\rightarrow\cdots \rightarrow u_m\text{I}u_j\text{D}v_s\text{I}v_k\rightarrow \cdots \rightarrow v_m\rightarrow\cdots\rightarrow v_t\rightarrow u_{\ell}\rightarrow\cdots\rightarrow u_1.
\end{align*}
Clearly, $v_s$ is a corner point of the cycle $C_{\text{combine}}$ (In fact, $u_m\text{I}u_j\text{D}v_s\text{I} v_k$ is an IDI path). Moreover, according to the construction the cyle $C_{\text{combine}}$ is a singular cycle with the smallest initial point and $v_s$ is the smallest point of singularity. Consequently, we get a new linear subdigraph $\gamma'\in \mathcal{C}$ so that $w(\gamma')=w(\gamma)$ and $c(\gamma')=c(\gamma)-1.$
\begin{figure}[H]
	\tiny
	\tikzstyle{ver}=[]
	\tikzstyle{vert}=[circle, draw, fill=black!100, inner sep=0pt, minimum width=4pt]
	\tikzstyle{vertex}=[circle, draw, fill=black!.5, inner sep=0pt, minimum width=4pt]
	\tikzstyle{edge} = [draw,thick,-]
	%\tikzstyle{edge_style} = [draw=black, line width=2, ultra thick]
	\tikzstyle{node_style} = [circle,draw=blue,fill=blue!20!,font=\sffamily\Large\bfseries]
	\centering
	%\tikzset{->,>=stealth', auto,node distance=1cm,
	%	thick,main node/.style={circle,draw,font=\sffamily\Large\bfseries}}
	\newcommand{\updownarrows}{\mathbin\uparrow\hspace{-.5em}\downarrow}
	\newcommand{\downuparrows}{\mathbin\downarrow\hspace{-.5em}\uparrow}
	
	\tikzset{->-/.style={decoration={
				markings,
				mark=at position #1 with {\arrow{>}}},postaction={decorate}}}
	\begin{tikzpicture}[scale=1]
	\tikzstyle{edge_style} = [draw=black, line width=2mm, ]
	\tikzstyle{node_style} = [draw=blue,fill=blue!00!,font=\sffamily\Large\bfseries]
	%EdgeStyle/.append style = {->, bend left} }
	\tikzset{
		LabelStyle/.style = { rectangle, rounded corners, draw,
			minimum width = 2em, fill = yellow!50,
			text = red, font = \bfseries },
		VertexStyle/.append style = { inner sep=5pt,
			font = \Large\bfseries},
		EdgeStyle/.append style = {->, bend left} }
	\tikzset{vertex/.style = {shape=circle,draw,minimum size=1.5em}}
	\tikzset{edge/.style = {->,> = latex'}}
	%%%%%%%%%%%%%%%%%%%%%%%%%%%%%%%%%%%%%%%%%%%%%%%%%%%%%%%%%%%%%%%%%%%%
	\node[] (a) at  (-2.7,0) {$\bf{u_{1}}$};
	\node[] (a) at  (2.25,-.18) {$\bf{u_{i}}$};
 \node[] (a) at  (5.9,.2) {$\bf{u_{m}}$};
	\node[] (a) at  (7.6,0) {$\bf{u_{j}}$};
	\node[] (a) at  (.18,-.25) {$\bf{u_{\ell}}$};
	\node[] (d) at  (5.4,-.25) {$\bf{u_r}$};
	%%%%%% fiiliing of node in first lsd %%%%%%
	\fill[blue!100!](-2.5,0) circle (.04);
	\fill[blue!100!](2.25,0) circle (.04);
 \fill[blue!100!](5.9,0) circle (.04);
	\fill[blue!100!](7.25,0) circle (.04);
	\fill[blue!100!](.25,0) circle (.04);
	\fill[blue!100!](5.25,0) circle (.04);
	\draw[dashed, blue] (.25,0)  to[bend left] (-2.5,0);
	\draw[dashed, blue] (-2.5,0)  to[bend left] (2.25, 0);
	\draw[dashed, blue] (2.25,0)  to[bend left] (5.9, 0);
 \draw[edge, blue] (5.9,0)  to[bend left] (7.25, 0);
	\draw[dashed,->,>=stealth, blue] (7.25,0)  to[bend left] (5.25, 0);
     \draw[edge, blue] (5.25,0)  to[bend left] (.25,0);
	%%%%%%%%%%%% red lsd %%%%%%%%%%%%%%
	\node[] (b) at  (1.35,-.2) {$\bf{v_s}$};
	\node[] (c) at  (4.4,-.2) {$\bf{v_t}$};
	\node[] (c) at  (5.1,.2) {$\bf{v_m}$};
	\node[] (c) at  (3,-.2) {$\bf{v_k}$};
\fill[red!100!](1.5,0) circle (.04);
\fill[red!100!](4.3,0) circle (.04);
\fill[red!100!](4.9,0) circle (.04);
\fill[red!100!](3,0) circle (.04);
\draw[edge, red] (4.3,0)  to[bend left] (1.5,0);
\draw[edge, red] (1.5, 0)  to[bend left] (3,0);
\draw[dashed, red] (3, 0)  to[bend left] (4.9,0);
\draw[dashed, ,->,>=stealth, red] (4.9, 0)  to[bend left] (4.3,0);
%%%%%%%%%%%%%%%%%%%%
\draw[<->, very thick] (2,-1.3) --(2, -2.3);
%%%%%%%%%%%%%%%%%%%%%%%%%%%%%%%%%%%%%% 2nd  lsd %%%%%%%%%%%%%%%%%%%%%%%%%%%%%%%%
\node[] (a) at  (-2.7,-3.25) {$\bf{u_{1}}$};
\node[] (a) at  (2.25,-3.2) {$\bf{u_{i}}$};
\node[] (a) at  (5.9,-2.8) {$\bf{u_{m}}$};
\node[] (a) at  (7.6,-3.25) {$\bf{u_{j}}$};
\node[] (a) at  (.18,-3.25) {$\bf{u_{\ell}}$};
\node[] (d) at  (5.4,-3.25) {$\bf{u_r}$};
%%%%%% fiiliing of node in first lsd %%%%%%
\fill[blue!100!](.25,-3) circle (.04);
\fill[blue!100!](5.25,-3) circle (.04);
\fill[blue!100!](5.9,-3) circle (.04);
\fill[blue!100!](7.25,-3) circle (.04);
\draw[dashed,blue] (.25,-3)  to[bend left] (-2.5,-3);
%\draw[dashed, blue] (-2.5,-3)  arc (180:30:10pt) to (2.25, -3);
%\draw (-2.5,-3)[dashed, looseness=.7, bend left, show control points] (2.25, -3);
\draw[dashed,blue] (-2.5,-3)  to[bend left] (2.25, -3);
\draw[dashed,blue] (2.25,-3)  to[bend left] (5.9, -3);
\draw[edge,blue] (5.9,-3)  to[bend left] (7.25, -3);
\draw[dashed,->,>=stealth, blue] (7.25,-3)  to[bend left] (5.25, -3);
\draw[edge,black] (4.3,-3)  to[bend left] (.25,-3);
%%%%%%%%%%%% red lsd %%%%%%%%%%%%%%
\node[] (b) at  (1.35,-3.2) {$\bf{v_s}$};
\node[] (c) at  (4.4,-3.2) {$\bf{v_t}$};
\node[] (c) at  (5.05,-2.83) {$\bf{v_m}$};
\node[] (c) at  (3,-3.2) {$\bf{v_k}$};
\fill[red!100!](3,-3) circle (.04);
\fill[black!100!](-2.5,-3) circle (.04);
\fill[black!100!](2.25,-3) circle (.04);
%\fill[black!100!](6.25,-3) circle (.04);
\fill[red!100!](1.5,-3) circle (.04);
\fill[red!100!](4.3,-3) circle (.04);
\fill[red!100!](4.9,-3) circle (.04);
\draw[edge] (5.25,-3)  to[bend left] (1.5,-3);
\draw[edge,red] (1.5, -3)  to[bend left] (3,-3);
\draw[dashed,red] (3, -3)  to[bend left] (4.9,-3);
\draw[dashed,->,>=stealth, red] (4.9, -3)  to[bend left] (4.3,-3);	
\end{tikzpicture}
\caption{Involution on linear subdigraphs.}
\label{fig:involution on stanly's thm}	
\end{figure}
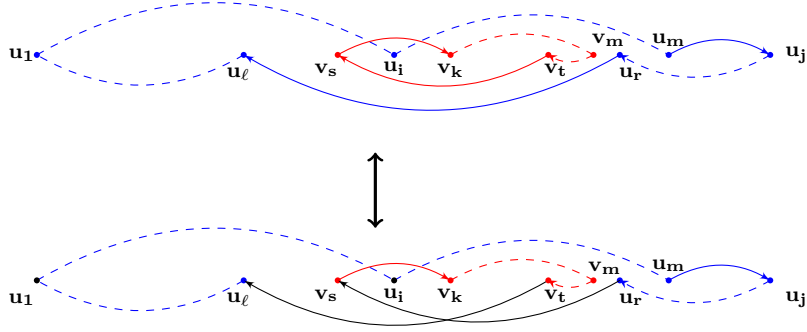
Conversely, let $v_s$ be a corner point in the cycle $C,$ where $C$ is the singular cycle with the smallest initial point $u_{1}$ (say). So, while traversing along the cycle starting from $u_1,$ we
eventually reach the point $v_s$ along a decreasing path $P_1$ connecting $u_r$ (say) and $v_s$ and then start moving along an increasing path $P_2$ connecting $v_s$ and $v_k$ (say). Now the path $P'$ from $v_k$ to $u_1$ has a maximal decreasing subpath say \[P'': v_m\rightarrow \cdots\rightarrow v_t\rightarrow u_{\ell}\rightarrow \cdots\rightarrow u_1.\] Clearly on the path $P'',$ there are two consecutive vertices namely $v_t$ and $u_{\ell}$ such that $u_{\ell}<v_s<v_t$ and $v_t\rightarrow u_{\ell}.$ First, we delete the arcs $v_t\rightarrow u_{\ell}$ and $u_r\rightarrow v_s.$ Then from $u_1$ to $u_r,$ follow the path along the cycle $C,$ then move from $u_r$ to $u_{\ell}$ by connecting the arc $u_r\rightarrow u_{\ell}$ (this is possible as $u_r>u_{\ell}$ implies that $a_{u_ru_{\ell}}\neq 0$ in $A(\Gamma_{k, R})).$ After that move from $u_{\ell}$ to $u_1$ by using the 
the path $u_{\ell}\rightarrow\cdots\rightarrow u_1$ along the cycle $C.$ Consequently, we complete a cycle \[C_1: u_1\rightarrow\cdots\rightarrow u_r\rightarrow u_{\ell}\rightarrow\cdots\rightarrow u_1.\] Now, from $v_s$ to $v_m$ moving along the increasing path $P_2$ and then from $v_m$ to $v_t$ traversing along the path $P''$. After that, moving from $v_t$ to $v_s$ by adding a directed edge $v_t\rightarrow v_s$ (this is possible as $v_s<v_t$ implies that $a_{v_tv_{s}}\neq 0$ in $A(\Gamma_{k, R})).$) Thus, we complete another cycle \[C_2: v_s\rightarrow \cdots\rightarrow v_m\rightarrow \cdots\rightarrow v_t\rightarrow v_s.\] As a result, we get a new linear subdigraph and according to the consturction, $v_s$ is the smallest point of singularity. Moreover, $v_s$ is an enclosed point. In this case, the number of cycles is increased by $1$ but weight is preserved. This completes the proof. 
\end{proof}

\subsection*{Data availability statement}
Availability of data and materials are not applicable.
\subsection*{Acknowledgment} I would like to thank Prof. Arvind Ayyer for his constant support and encouragement. The theorem is checked by ``Sage''. We thank the authors for generously releasing Sage as an open-source package. 
\bibliographystyle{amsplain}
\bibliography{gen-inv-lcp}
\end{document}